\documentclass[12pt]{article}

\usepackage{amsmath,amssymb,amsthm,amsfonts,fancyhdr}
\newtheorem{definition}{Definition}
\newtheorem{thm}{Theorem}

\newtheorem{cor}{Corollary}
\newtheorem{rem}{Remark}

\theoremstyle{definition}
\newtheorem{exm}{Example}

\newcommand{\N}{\mathbb{N}}
\newcommand{\Z}{\mathbb{Z}}
\newcommand{\R}{\mathbb{R}}
\newcommand{\norm}[1]{||#1||}
\newcommand{\abs}[1]{|#1|}
\newcommand{\e}{\varepsilon}

\begin{document}

\title{Bounded solutions and asymptotic stability to nonlinear second-order neutral difference equations with quasi-differences}
\author{Magdalena Nockowska--Rosiak\thanks{%
Lodz University of Technology, Poland, email: magdalena.nockowska@p.lodz.pl}}
\maketitle

\begin{abstract}
This work is devoted to the study of the nonlinear second-order neutral difference equations with quasi-differences of the form
$$
\Delta \left( r_{n} \Delta \left( x_{n}+q_{n}x_{n-\tau}\right)\right)= a_{n}f(x_{n-\sigma})+b_n
$$
with respect to $(q_n)$. For $q_n\to1$, $q_n\in(0,1)$ the standard fixed point approach is not sufficed to get the existence of the bounded solution, so we combine this method with an approximation technique to achieve our goal. Moreover, for $p\ge 1$ and $\sup|q_n|<2^{1-p}$ using Krasnoselskii's fixed point theorem we obtain sufficient conditions of the existence of the solution which belongs to $l^p$ space.
\end{abstract}

\section{Introduction}
Difference equations are used in mathematical models in diverse areas such as economy, biology, computer science, see, for example \cite{Agarwal}, \cite{Elaydi}. In the past thirty years, oscillation, nonoscillation, the asymptotic behaviour and existence of bounded solutions to many types second-order difference equation have been widely examined, see for example \cite{Agarwal2003}, \cite{Cheng1999}, \cite{Dosly}, \cite{Galewski}, \cite{Gou}, \cite{JS}, \cite{Lalli1994}, \cite{Lalli1992}, \cite{Lou}, \cite{Meng}, \cite{Migda2005}, \cite{Migda2014}, \cite{S1}, \cite{S2}, \cite{S3}, \cite{Schmeidel2013}, \cite{Schmeidel2012}, \cite{TKP1}, and references therein.
\par
The second-order difference equation with quasi-difference of the form
$$
\Delta \left( r_{n} \Delta \left( x_{n}+q_{n}x_{n-\tau}\right)\right)= F(n,x_{n-\sigma}) 
$$
is studied in the literature with respect to a sequence $(q_n)$. The fixed point theory is the standard technique to prove the existence of the bounded solution to the considered problem with constant $(q_n)$ and $(q_n)$ which is separated from 1. Let us present short overview of papers which deal with this problem. By using Banach's fixed point theorem, Jinfa \cite{Jinfa} and Liu et al. \cite{Liu2009} investigated the nonoscillatory solution to the second-order neutral delay difference equation with any constant coefficients $q_n$ it means the equation
$$
\Delta \left( r_{n} \Delta \left( x_{n}+qx_{n-\tau}\right)\right)+ f(n, x_{n-d_{1n}},\ldots,x_{n-d_{kn}})=c_n.
$$
In fact Liu et al. \cite{Liu2009} proved the existence of uncountable many bounded nonoscillatory solutions for the above problem under Lipschitz continuity condition. By Leray-Schauder type of condensing operators Agarwal et al. \cite{Agarwal2003} examined the existence of a nonoscillatory solution to the problem
$$
\Delta \left( r_{n} \Delta \left( x_{n}+qx_{n-\tau}\right)\right)+ F(n+1, x_{n+1-\sigma})=0,
$$
where $q\in\R\setminus\{\pm1\}$. Liu et al. \cite{Liu2011} discussed the existence of uncountable many bounded positive solutions to
$$
\Delta \left( r_{n} \Delta \left( x_{n}+b_nx_{n-\tau}-c_n\right)\right)+ f(n, x(f_1(n)),\ldots,x(f_k(n)))=d_n,
$$
where $\sup_{n\in\N}b_n=b^\star$, $b^\star\neq1$ or $\inf_{n\in\N}b_n=b_\star$, $b_\star\neq-1$ by Krasnoselskii's fixed point theorem.
\par
On the other hand, Petropoulous and Siafarikas considered different types of difference equations in the Hilbert space, see, \cite{PS2001}, \cite{PS2004}, \cite{PS2006}. Moreover, the functional-analytical method to a general nonautonomous difference equations of the form $x_{k+1}=f_k(x_k,x_{k+1})$ was considered by P\"otzsche and Ey \cite{EP2008}, P\"otzsche \cite{P2009}. This approach allows better characterize solutions to difference equations.
\par
In this paper we study the following second-order neutral difference equation with quasi-difference
\begin{equation}\label{problem}
\Delta \left( r_{n} \Delta \left( x_{n}+q_{n}x_{n-\tau}\right)\right)= a_{n}f(x_{n-\sigma})+b_n, 
\end{equation}
where $\tau\in\N\cup\{0\}$, $\sigma\in\Z$, $a, b, q:{\mathbb{N}}\rightarrow {\mathbb{R}}$, $r:{\mathbb{N}}\rightarrow {\mathbb{R}}\setminus \{0\}$, and $f:{\mathbb{R}}\rightarrow {\mathbb{R}}$ is a locally Lipschitz function. If the sequence $(q_n)$ is convergent to 1, then fixed point approach can not be applied to solve the studied problem, because Krasnoselskii's fixed point theorem need two operators which one of them is a contraction. To overcome the limitation of this method we combine this approach with the approximation technique under the additionally assumption $q_n\in(0,1)$. The approximation approach in this type of difference equations and the case when $(q_n)$ is convergent to 1 has not been discussed so far, to our knowledge. Moreover, in the case $p\ge 1$ and $\sup|q_n|<2^{1-p}$ we establish sufficient conditions of the existence the solution to \eqref{problem} which belongs to $l^p$ space. To get our result Krasnoselskii's fixed point theorem is used.

\section{ Preliminaries}

Throughout this paper, we assume that $\Delta$ is the forward difference operator, ${\mathbb{N}}_{k}:=\left\{k,k+1,\dots \right\} $ where $k$ is a given positive integer, $\N_1=\N$ and ${\mathbb{R}}$ is a set of all real numbers. 
\par
Let $k\in\N$. We consider the Banach space $l^{\infty}_{k}$ of all real bounded sequences $x\colon {\mathbb{N}}_{k}\rightarrow {\mathbb{R}}$ equipped with the standard supremum norm, i.e. 
$$
\Vert x\Vert =\sup_{n\in {\mathbb{N}}_{k}}|x_{n}|, \text{ for }x=(x_n)_{n\geq k}\in \ l^{\infty}_k.
$$
\begin{definition}\cite{Cheng}
A subset $A$ of $l^\infty_k$ is said to be uniformly Cauchy if for every $\varepsilon>0$ there exists $n_0\in\N_k$ such that $\abs{x_i-x_j}<\varepsilon$
for any $i,j\geq n_0$ and $x=(x_n)\in A$.
\end{definition}
\begin{thm}\cite{Cheng}\label{uni}
A bounded, uniformly Cauchy subset of $l^\infty_k$ is relatively compact.
\end{thm}
For a real $p\ge 1$ we define $l^p$ the Banach space of $p$-summable sequences as follows
$$
l^p:=\{x\colon {\mathbb{N}}\rightarrow {\mathbb{R}}: \sum^\infty_{n=1}|x(n)|^p<\infty\}
$$
with the standard norm, i.e.
$$
\Vert x\Vert_{l^p} =\left(\sum^\infty_{n=1}|x(n)|^p\right)^{1/p}.
$$
The relative compactness criterion in $l^p$ is given in the following theorem
\begin{thm}\label{costara}(\cite{costara},  p.106)
Let $p\in[1,\infty)$. A subset $A$ of $l^p$ is relatively compact if and only if $A$ is bounded and 
$$
\lim\limits_{l\to\infty}\sup_{x=(x_n)\in A}\sum^\infty_{n=l}|x_n|^p=0.
$$
\end{thm}
To get main results of this paper we use Krasnoselskii's fixed point theorem of the form.
\begin{thm}\label{Kras}(\cite{Zeidler}, 11.B p. 501)
Let $X$ be a Banach space, $B$ be a bounded, closed, convex subset of $X$ and $S, G:B\to X$ be mappings such that $Sx+Gy\in B$ for any $x,y\in B$. If $S$ is a contraction and $G$ is a compact, then the equation
$$
Sx+Gx=x
$$ 
has a solution in $B$.
\end{thm}
To use the approximation technique we need the following the Banach-Alaoglu theorem.
\begin{thm}\cite{Rudin}\label{Alaoglu} 
If $X$ is Banach space and $S^\star=\{x^\star\in X^\star:\norm{x^\star}\leq 1\}$, then $S^\star$ is weak$^\star$-compact. 
\end{thm}
Let us close the preliminaries paragraph by definitions of different types of solution to \eqref{problem}. By a solution to equation \eqref{problem} we mean a sequence $x:{\mathbb{N}}_{k}\rightarrow {\mathbb{R}}$ which satisfies \eqref{problem} for every $n\in {\mathbb{N}}_{k}$ for some $k\geq\max\{\tau,\sigma\}$. By a full solution to equation \eqref{problem} we mean a sequence $x:{\mathbb{N}}_{\max\{\tau,\sigma\}}\rightarrow {\mathbb{R}}$ which satisfies \eqref{problem} for every $n\geq\max\{\tau,\sigma\}$. For $p\ge1$, a solution $x$ to \eqref{problem} is said to $l^p$ solution, if $x\in l^p$.

\section{The existence of bounded solutions respect to sequence $(q_n)$}

In this section, sufficient conditions for the existence of a bounded solution to equation \eqref{problem} respect to values of sequence $(q_n)$ are derived.

In this section, unless otherwise note, we assume $\tau\in\N\cup\{0\}$, $\sigma\in\Z$, $a, b, q:{\mathbb{N}}\rightarrow {\mathbb{R}}$, $r:{\mathbb{N}}\rightarrow {\mathbb{R}}\setminus \{0\}$ and $f:{\mathbb{R}}\rightarrow \mathbb{R}$. 

\begin{thm}\label{l-infty}
Assume that
\begin{itemize}
\item[$(H_{fl})$] $f$ is a locally Lipschitz function;
\item[$(H_{s})$] $\sum\limits_{s=1}^{\infty } \left\vert\frac{1}{r_{s}}\right\vert\sum\limits_{t=s}^{\infty }\left\vert a_{t}\right\vert <+\infty,$\quad $\sum\limits_{s=1}^{\infty } \left\vert\frac{1}{r_{s}}\right\vert\sum\limits_{t=s}^{\infty }\left\vert b_{t}\right\vert <+\infty,$
\item[$(H_{q})$] $\sup\limits_{n\in\N}|q_n|=q^\star<1$.
\end{itemize} 
Then, the equation \eqref{problem} possesses a bounded solution.
\end{thm}
\begin{proof}
Let $M>0$. From the continuity of $f$ on $[-M,M]$ we get the existence of $Q>0$ such that
$$
|f(x)|\leq Q, \ \textrm{for} \ x\in[-M, M].
$$
By $(H_s)$ there exists $n_0>\beta:=\max\{\tau,\sigma\}$ such that
\begin{equation}\label{szereg-n0}
\sum^{\infty}_{s=n_0}\left\vert\frac{1}{r_s}\right\vert\sum^{\infty}_{t=s}\left(|a_t|Q+|b_t|\right)<(1-q^\star)M.
\end{equation}
We consider the Banach space $l^\infty_1$ and its subset
$$
A_{n_0}=\left\{x=(x_n)_{n\in\N_1}\in l^\infty_1: x_1=\ldots=x_{n_0+\beta-1}=0, \ |x_n|\leq M, n\geq n_0+\beta\right\}.
$$ 
Observe that $A_{n_0}$ is a nonempty, bounded, convex and closed subset of $l^{\infty}_1$.
\\Define two mappings $T_1,T_2\colon l^{\infty }_1\rightarrow l^{\infty}_1$ as follows 
$$
(T_1x)_{n}=
\begin{cases}
0, &\text{for}\ 1\leq n<n_{0}+\beta
\\ 
-q_{n}x_{n-\tau}, &\text{for}\ n\geq n_{0}+\beta
\end{cases}
$$
$$
(T_2x)_{n}=
\begin{cases}
0, &\text{for}\ 1\leq n<n_{0}+\beta
\\ 
\sum\limits_{s=n}^{\infty }\frac{1}{r_{s}}\sum\limits_{t=s}^{\infty }\left( a_{t}f(x_{t-\sigma})+b_t \right), &\text{for}\ n\geq n_{0}+\beta.
\end{cases}
$$
Our next goal is to check assumptions of Theorem \ref{Kras} - Krasnoselskii's fixed point.
\\Firstly, we show that $T_1x+T_2y\in A_{n_0}$ for $x,y\in A_{n_0}$. Let $x,y\in A_{n_0}$. For $n<n_{0}+\beta$ $(T_1x+T_2y)_n=0$. For $n\geq n_{0}+\beta$ from assumption $(H_q)$ and \eqref{szereg-n0} we get
$$
\left\vert(T_1x+T_2y)_n\right\vert\leq \left\vert q_nx_{n-\tau}\right\vert +\sum^{\infty}_{s=n}\left\vert\tfrac{1}{r_s}\right\vert\sum^\infty_{t=s}\left(|a_t|Q+|b_t|\right)\leq q^\star M+(1-q^\star) M=M.
$$
It is easy to see that
$$
\norm{T_1x-T_1y}\leq q^\star \norm{x-y},\ \textrm{for} \ x,y\in A_{n_0},
$$
so that $T_1$ is a contraction. 
\\To prove the continuity of $T_2$, we note assumption $(H_{fl})$ implies that $f$ is Lipschitz function on $[-M,M]$, say, with constant $L>0$, which means
$$
\abs{f(u)-f(v)}\leq L\abs{u-v},\ \textrm{for} \ u,v\in [-M,M].
$$
Hence
$$
\norm{T_2x-T_2y}\leq L\left(\sum\limits_{s=1}^{\infty } \left\vert\frac{1}{r_{s}}\right\vert\sum\limits_{t=s}^{\infty }\left\vert a_{t}\right\vert \right)\norm{x-y}, \ \textrm{for} \ x,y\in A_{n_0}.
$$
Actually, we prove that $T_2$ is the Lipschitz operator.
\\Now we show that $T_2(A_{n_0})$ is uniformly Cauchy. Let $\e>0$. From $(H_s)$ we get the existence of $n_\e\in\N$ such that 
$$
2\sum^{\infty}_{s=n_\e}\left\vert\frac{1}{r_s}\right\vert\sum^{\infty}_{t=s}\left(|a_t|Q+|b_t|\right)<\e.
$$
For $m>n\geq n_\e\geq n_0+\beta$ and for $x\in A_{n_0}$ we have 
\begin{align*}
&\abs{(T_2x)_n-(T_2x)_m}=\left\vert\sum^{\infty}_{s=n}\frac{1}{r_s}\sum^{\infty}_{t=s}\left(a_tf(x_{t-\sigma})+b_t\right)-\sum^{\infty}_{s=m}\frac{1}{r_s}\sum^{\infty}_{t=s}\left(a_tf(x_{t-\sigma})+b_t\right)\right\vert
\\
&\le2\sum^{\infty}_{s=n_\e}\left\vert\frac{1}{r_s}\right\vert\sum^{\infty}_{t=s}\left(|a_t|Q+|b_t|\right)<\e.
\end{align*}
Since $T_2(A_{n_0})$ is uniformly Cauchy and bounded then by Theorem \ref{uni} $T_2(A_{n_0})$ is relatively compact in $l^\infty$ which means that $T_2$ is a compact operator.  

From Krasnosielskii's theorem we get that there exists $x=(x_n)_{n\in\N_1}$ the fixed point of $T_1+T_2$ on $A_{n_0}$. Applying operator $\Delta$ to both sides of the above equation and multiplying by $r_n$ and applying operator $\Delta$ second time for $n\ge n_0+\beta$ we get $x=(x_n)_{n\in\N_{n_0+\beta}}$ is the solution to \eqref{problem}. 
\end{proof}
\begin{cor}\label{l-infty-drugie}
Assume that
\begin{itemize}
\item[$(H_{fl})$] $f$ is a locally Lipschitz function;
\item[$(H'_{s})$] $\sum\limits_{s=\sigma+1}^{\infty } \left\vert\frac{1}{r_{s}}\right\vert\sum\limits_{t=\sigma}^{s-1 }\left\vert a_{t}\right\vert <+\infty,$\quad $\sum\limits_{s=\sigma+1}^{\infty } \left\vert\frac{1}{r_{s}}\right\vert\sum\limits_{t=\sigma}^{s-1 }\left\vert b_{t}\right\vert <+\infty,$
\item[$(H_{q})$] $\sup\limits_{n\in\N}|q_n|=q^\star<1$.
\end{itemize} 
Then, the equation \eqref{problem} possesses a bounded solution.
\end{cor}
\begin{proof}
The proof is analogous to the proof of the Theorem \ref{l-infty} with 
$$
(T_2x)_{n}=
\begin{cases}
0, &\text{for}\ 1\leq n<n_{0}+\beta
\\ 
-\sum\limits_{s=n}^{\infty }\frac{1}{r_{s}}\sum\limits_{t=\sigma}^{s-1 }\left( a_{t}f(x_{t-\sigma})+b_t \right), &\text{for}\ n\geq n_{0}+\beta
\end{cases}
$$
where for any $M>0$ there exist $Q>0$ and $n_0>\beta:=\max\{\tau,\sigma\}$ such that
$$
\sum^{\infty}_{s=n_0}\left\vert\frac{1}{r_s}\right\vert\sum^{s-1}_{t=\sigma}\left(|a_t|Q+|b_t|\right)<(1-q^\star)M.
$$
\end{proof}

\begin{rem}
Assumptions $(H_{s})$ and $(H'_{s})$ are not comparable. Indeed, let us consider sequences $(a_n), (b_n), (r_n)$, where $a_n=\tfrac{1}{(2n-1)(2n+1)}$, $b_n=0$ $r_n=\sqrt{n}$ for $n\in\N$, (see \cite{N2016}). Then
$$
\sum^{\infty}_{s=1}\left|\tfrac{1}{r_{s}}\right|\sum^{\infty}_{t=s}|a_t|=\sum^{\infty}_{s=1}\tfrac{1}{\sqrt{s}}\sum^{\infty}_{t=s}\tfrac{1}{(2t-1)(2t+1)}=\sum^{\infty}_{s=1}\tfrac{1}{2\sqrt{s}(2s-1)}<\infty,
$$
so assumption $(H_s)$ of Theorem \ref{l-infty} is satisfied. Assumption $(H'_{s})$ of Corollary \ref{l-infty-drugie} is not fulfil because  
$$
\sum^{\infty}_{s=1}\left|\tfrac{1}{r_{s}}\right|\sum^{s-1}_{t=1}|a_t|=\sum^{\infty}_{s=1}\tfrac{1}{\sqrt{s}}\sum^{s-1}_{t=1}\tfrac{1}{(2t-1)(2t+1)}=\sum^{\infty}_{s=1}\tfrac{(s-1)}{\sqrt{s}(2s-1)}=\infty.
$$ 
On the other hand, let us consider sequences $(a'_n), (b'_n), (r'_n)$, where $a'_n=n$, $b'_n=0$ $r'_n=2^n$ for $n\in\N$. Then
$$
\sum^{\infty}_{s=1}\left|\tfrac{1}{r_{s}}\right|\sum^{s-1}_{t=1}|a_t|=\sum^{\infty}_{s=1}2^{-s}\sum^{s-1}_{t=1}t=\sum^{\infty}_{s=1}2^{-s-1}s(s-1)<\infty.
$$
which means assumption $(H'_{s})$ of Corollary \ref{l-infty-drugie} is satisfied. To see that assumption $(H_{s})$ of Theorem \ref{l-infty} is not fulfilled notice that $(H_s)$ implies that $\sum^\infty_{t=0}|a'_t|<\infty,$ which give a contradiction for $(a'_n)$.
\end{rem}
\begin{cor}
If in Theorem \ref{l-infty} and  Corollary \ref{l-infty-drugie}  we additionally assume
\begin{itemize}
\item[$(H'_{0})$] $\tau,\sigma\in\N\cup\{0\}$, $\tau>\sigma$ and \quad $q_n\neq0$ for $n\in\N$.
\end{itemize} 
Then, the equation \eqref{problem} possesses a bounded full solution.
\end{cor}
\begin{proof}
We find previous $n_0$ terms of sequence $x$ by formula 
$$
x_{n-\tau}=\frac{1}{q_{n}}\left( -x_{n}+\sum\limits_{s={n}}^{\infty} \frac{1}{r_{s}}\sum\limits_{t=s}^{\infty }\left(a_{t}f(x_{t-\sigma})+b_{t}\right)\right) ,
$$
or
$$
x_{n-\tau}=\frac{1}{q_{n}}\left( -x_{n}+\sum\limits_{s={n}}^{\infty} \frac{1}{r_{s}}\sum\limits_{t=\sigma}^{s-1 }\left(a_{t}f(x_{t-\sigma})+b_{t}\right)\right) ,
$$
starting with putting $n:=n_0+2\tau-1$. 
\end{proof}

Using the same technique we get the following result.
\begin{thm}
Assume that
\begin{itemize}
\item[$(H_0)$] $\tau,\sigma\in\N\cup\{0\}$, $\tau>\sigma$,
\item[$(H_{fl})$] $f$ is a locally Lipschitz function;
\item[$(H_{s})$] $\sum\limits_{s=1}^{\infty } \left\vert\frac{1}{r_{s}}\right\vert\sum\limits_{t=s}^{\infty }\left\vert a_{t}\right\vert <+\infty,$\quad $\sum\limits_{s=1}^{\infty } \left\vert\frac{1}{r_{s}}\right\vert\sum\limits_{t=s}^{\infty }\left\vert b_{t}\right\vert <+\infty,$
\item[$(H^1_{q})$] $\inf\limits_{n\in\N}q_n=q^\star>1$.
\end{itemize}
Then, there exists a bounded full solution to \eqref{problem}.
\end{thm}
\begin{proof}
The proof is similar to the proof of the Theorem \ref{l-infty} 
with operators
$$
(T_1x)_{n}=
\begin{cases}
0, &\text{for}\ 1\leq n<n_{0}
\\ 
-\tfrac{1}{q_{n+\tau}}x_{n+\tau}, &\text{for}\ n\geq n_{0}
\end{cases}
$$
$$
(T_2x)_{n}=
\begin{cases}
0, &\text{for}\ 1\leq n<n_{0}
\\ 
\tfrac{1}{q_{n+\tau}}\sum\limits_{s=n+\tau}^{\infty }\frac{1}{r_{s}}\sum\limits_{t=s}^{\infty }\left( a_{t}f(x_{t-\sigma})+b_t \right), &\text{for}\ n\geq n_{0},
\end{cases}
$$
where for any $M>0$ there exist $Q>0$ and $n_0>\beta:=\max\{\tau,\sigma\}$ such that
$$
\sum^{\infty}_{s=n_0}\left\vert\frac{1}{r_s}\right\vert\sum^{\infty}_{t=s}\left(|a_t|Q+|b_t|\right)<(1-\tfrac{1}{q^\star})M.
$$
\end{proof}

Now we use an approximation technique to get our main result of this section.
\begin{thm}\label{zb-1}
Assume that
\begin{itemize}
\item[$(H_0)$] $\tau,\sigma\in\N\cup\{0\}$, $\tau>\sigma$,
\item[$(H_{fb})$] $f$ is a locally Lipschitz function and $f$ is a bounded function with constant $P$;
\item[$(H_{sb})$] there exist $C\in(0,1)$ and increasing sequence $(w_k)_{k\in\N}\subset(0,1)$ with $\sum^\infty_{k=1}(1-w_k)<\infty$ such that
$$
\sum\limits_{s=k}^{\infty }\left\vert\frac{1}{r_{s}}\right\vert\sum\limits_{t=s}^{\infty }(\left\vert a_{t}\right\vert P+\left\vert b_{t}\right\vert )\in O\left((1-w_k)\left(Cw_k\right)^k\right);
$$ 


\item[$(H_{q=1})$] $q_n\in(0,1),\ n\in\N$, $\lim\limits_{n\to\infty}q_n=1$, $\inf\limits_{n\in\N}q_n>0$.
\end{itemize} 
Then, there exists a bounded full solution 
to equation \eqref{problem}.
\end{thm}

\begin{proof}
For any $k\in\N$, let us consider an auxiliary problem
\begin{equation}\label{problem-pom}
\Delta \left( r_{n} \Delta \left( x_{n}+w_kq_{n}x_{n-\tau}\right)\right)= a_{n}f(x_{n-\sigma})+b_n,
\end{equation}
where $(w_k)$ is the sequence satisfying $(H_{sb})$. It is obvious that
$$
\sup\{w_kq_{n}: n\in\N\}=w_k<1.
$$
Without loss of generality we can assume that  
$$
\inf\{q_{n}: n\in\N\}>C,
$$
where $C$ is the constant from assumption $(H_{sb})$. By $(H_{sb})$ there exist $k_0\in\N$, $D>0$ such that for any $k\ge k_0$ 
\begin{equation}\label{szereg-osz}
\sum\limits_{s=k}^{\infty }\left\vert\frac{1}{r_{s}}\right\vert\sum\limits_{t=s}^{\infty }\left\vert (a_{t}\right\vert P+\left\vert b_{t}\right\vert )\le D(1-w_k)\left(Cw_k\right)^k.
\end{equation}
From Theorem \ref{l-infty}, \eqref{problem-pom} possesses a bounded solution $x^k=(x^k_n)_{n\in\N_{n_k+\tau}}$ for some $n_k\in\N$. Moreover, by the proof of the Theorem \ref{l-infty} we see that \eqref{szereg-osz} implies \eqref{szereg-n0} with $M_k=D(Cw_k)^k$. From the Theorem \ref{l-infty} we get $x^k=(x^k_n)_{n\in\N_{_{n_k+\tau}}}\in l^\infty_{n_k+\tau}$ as the fixed point of the $T_1+T_2$ on $A_{n_k}$. 
By \eqref{szereg-osz} $n_k:=k$ for $k\ge k_0$. It means that for $k\ge k_0$, $x^k=(x^k_n)$ solve \eqref{problem-pom} for $n\ge k+\tau$ and  $|x_n^k|\le D(Cw_k)^k$ for $n\ge k+\tau$. We find previous $k$ terms of sequence $(x^k_n)_{n\ge \tau}$ by formula 
\begin{equation*}
x_{n-\tau}^k=\frac{1}{w_kq_{n}}\left( -x^k_{n}+\sum\limits_{s={n}}^{\infty} \frac{1}{r_{s}}\sum\limits_{t=s}^{\infty }\left(a_{t}f(x_{t-\sigma})+b_{t}\right)\right).
\end{equation*}
Putting $n:=n_0+2\tau-1=k+2\tau-1$ to above we get
\begin{align*}
&\abs{x^k_{n_0+\tau-1}}=\abs{x^k_{k+\tau-1}}\le \frac{1}{Cw_k}\left( D(Cw_k)^k+\sum\limits_{s={k+2\tau-1}}^{\infty} \frac{1}{|r_{s}|}\sum\limits_{t=s}^{\infty }(|a_{t}|P+|b_{t}|)\right)
\\
&\le \frac{1}{Cw_k}\left( D(Cw_k)^k+\sum\limits_{s={k}}^{\infty} \frac{1}{|r_{s}|}\sum\limits_{t=s}^{\infty }(|a_{t}|P+|b_{t}|)\right)\le D(1+(1-w_k))(Cw_k)^{k-1}.
\end{align*}
For $k\ge 2$
\begin{align*}
&\abs{x^k_{n_0+\tau-2}}=\abs{x^k_{k+\tau-2}}\le \frac{1}{Cw_k}\left( |x_{k+2\tau-2}|+\sum\limits_{s=k+2\tau-2}^{\infty} \frac{1}{|r_{s}|}\sum\limits_{t=s}^{\infty }(|a_{t}|P+|b_{t}|)\right)
\\
&\le\begin{cases}
\frac{1}{Cw_k}\left( D(1+(1-w_k))(Cw_k)^{k-1}+\sum\limits_{s=k}^{\infty} \frac{1}{|r_{s}|}\sum\limits_{t=s}^{\infty }(|a_{t}|P+|b_{t}|)\right) &\mbox{for} \ \tau=1
\\
\frac{1}{Cw_k}\left( D(Cw_k)^{k}+D(1-w_k)(Cw_k)^{k}\right) &\mbox{for} \ \tau\ge2
\end{cases}
\\
&\le\begin{cases}
D(1+2(1-w_k))(Cw_k)^{k-2} &\mbox{for} \ \tau=1
\\
D(1+(1-w_k))(Cw_k)^{k-1}&\mbox{for} \ \tau\ge2
\end{cases}
\end{align*}
We give the estimation of $|x^k_n|$ for the case $\tau=1$, $\sigma=0$. The other case are analogous and are left to the reader. Indeed,
\begin{align*}
&|x^k_{k-2}|\le(Cw_k)^{-1}\left(|x^k_{k-1}|+\sum\limits_{s={k-1}}^{\infty} \frac{1}{|r_{s}|}\sum\limits_{t=s}^{\infty }(|a_{t}|P+|b_{t}|)\right)
\\
&\le(Cw_k)^{k-3}\left(D+2D(1-w_k)\right)+D\tfrac{(Cw_{k-1})^{k-1}}{Cw_k}(1-w_{k-1})
\\
&\le(Cw_k)^{k-3}(D+2D(1-w_k))+D\tfrac{(Cw_k)^{k-1}}{Cw_k}(1-w_{k-1})
\\
&\le(Cw_k)^{k-3}(D+2D(1-w_k)+D(1-w_{k-1}))\le D+2D(1-w_k)+D(1-w_{k-1}).
\end{align*}
By induction for $i=3,\ldots,k-k_0+1$ 
\begin{align*}
&|x^k_{k-i}|\le D(Cw_k)^{k-i-1}\left(1+2(1-w_k)+\sum\limits_{j=k-i+1}^{k-1}(1-w_j)\right)
\\
&\le 2D+D\sum^{\infty}_{i=1}(1-w_{i}).
\end{align*}
Moreover,
$$
|x^k_{k_0-2}|\le\tfrac{1}{Cw_1}\left(2D+D\sum^{\infty}_{i=1}(1-w_{i})+\sum\limits_{s=k_0-1}^{\infty}\frac{1}{|r_{s}|}\sum\limits_{t=s}^{\infty }(|a_{t}|P+|b_{t}|)\right)
$$
and by induction
$$
|x^k_{1}|\le\tfrac{1}{Cw_1}\left(2D+D\sum^{\infty}_{i=1}(1-w_{i})+\sum\limits_{j=2}^{k_0-1}\sum\limits_{s=j}^{\infty}\frac{1}{|r_{s}|}\sum\limits_{t=s}^{\infty }(|a_{t}|P+|b_{t}|)\right).
$$
Hence, for $n\in\N_{\tau}$, $k\ge k_0$
$$
|x^k_n|\le\tfrac{1}{Cw_1}\left(2D+D\sum^{\infty}_{i=1}(1-w_{i})+\sum\limits_{j=2}^{k_0-1}\sum\limits_{s=j}^{\infty}\frac{1}{|r_{s}|}\sum\limits_{t=s}^{\infty }(|a_{t}|P+|b_{t}|)\right).
$$
It means that  the sequence $(x^k)_{k\ge k_0}$ is bounded in $l^\infty_{\tau}$. Since $(l^1_\tau)^{\star}=l^\infty_\tau$, then from the Banach-Alaoglu theorem we get that there exists $(x^{k_l})_{l\in\N}\subset(x^{k})_{k\in\N}$ which convergent on its coordinates. It means that there exists $\overline{x}=(\overline{x}_n)_{n\in\N_{\tau}}\in l^\infty_{\tau}$ such that
$$
\lim_{l\to\infty}x^{k_l}_n=\overline{x}_n, \ \textrm{for} \ n\in\N_{\tau}
$$
and
\begin{equation}\label{ogr-x}
|\overline{x}_n|\le\tfrac{1}{Cw_1}\left(2D+D\sum^{\infty}_{i=1}(1-w_{i})+\sum\limits_{j=1}^{k_0-1}\sum\limits_{s=j}^{\infty}\frac{1}{|r_{s}|}\sum\limits_{t=s}^{\infty }(|a_{t}|P+|b_{t}|)\right), 
\end{equation}
for $n\in\N_{\tau}$. To get our results we pass with $l\to\infty$ in 
\begin{equation}\label{apr}
x^{k_l}_n+w_{k_l}q_nx^{k_l}_{n-\tau}=\sum\limits_{s={n}}^{\infty} \frac{1}{r_{s}}\sum\limits_{t=s}^{\infty }\left(a_{t}f(x^{k_l}_{t-\sigma})+b_{t}\right), \ \mbox{for} \ n\geq \tau. 
\end{equation}
It is easy to see that for $n\geq\tau$ we have
$$
\lim_{l\to\infty}\left(x^{k_l}_n+w^{k_l}q_nx^{k_l}_{n-\tau}\right)=\overline{x}_{n}+q_n\overline{x}_{n-\tau}.
$$ 
From Lebesgue's dominated convergence theorem and the continuity of $f$ we get
\begin{align*}
&\lim_{l\to\infty}\left(\sum^\infty_{s=n}\tfrac{1}{r_s}\sum^\infty_{t=s}\left(a_tf(x^{k_l}_{t-\sigma})+b_t\right)\right)=\sum^\infty_{s=n}\tfrac{1}{r_s}\sum^\infty_{t=s}\left(a_t\left(\lim_{l\to\infty} f(x^{k_l}_{t-\sigma})\right)+b_t\right)=
\\
&\sum^\infty_{s=n}\tfrac{1}{r_s}\sum^\infty_{t=s}\left(a_t f(\overline{x}_{t-\sigma})+b_t\right).
\end{align*}
From \eqref{apr} we get that
$$
\overline{x}_n+q_n\overline{x}_{n-\tau}=\sum\limits_{s={n}}^{\infty} \frac{1}{r_{s}}\sum\limits_{t=s}^{\infty }\left(a_{t}f(\overline{x}_{t-\sigma})+b_{t}\right), 
$$
for $n\geq\tau$. Applying operator $\Delta$ to both sides of the above equation and multiplying by $r_n$ and applying operator $\Delta$ second time we get
$$
\Delta\left(r_n\Delta(\overline{x}_n+q_n\overline{x}_{n-\tau})\right)=a_nf(\overline{x}_{n-\sigma})+b_n, \ \textrm{for} \ n\geq\tau.
$$
From \eqref{ogr-x} we get that $(\overline{x}_n)_{n\geq\tau}$ is the bounded full solution to \eqref{problem}.
\end{proof}

Now, we present an example of equation which can be considered by our method.

\begin{exm}\label{1}
The following problem 
\begin{equation}\label{prob}  
\Delta\left( \left( -1\right) ^{n}\Delta \left( x_{n}+\left(1-\tfrac{1}{2^n}\right)x_{n-3}\right)\right)=\tfrac{3}{4}\tfrac{1}{2^{n}}(\sin x_{n-1})^6,\ n\ge 3
\end{equation}
with $\tau=3$, $\sigma=1$, $r_{n}=(-1)^{n}$, $q_{n}=1-\tfrac{1}{2^n}$, $a_{n}=\frac{3}{4}\tfrac{1}{2^{n}}$, $b_n=0$, $n\ge 1$ and an bounded $f(x)=(\sin x)^6$ fulfil assumptions of the Theorem \ref{zb-1}. Indeed, because $f$ is a locally Lipschitz function as $f\in C^1(\R)$, we have to check only $(H_{sb})$. For $C=9/10$ and $w_k=1-(5/8)^{k}$, $k\ge 1$ and $P=1$ we get for there exists $k_0$ such that for any $k\ge k_0$
$$
3\left(\tfrac{8}{9}\right)^k<\left(1-\left(\tfrac{5}{8}\right)^k\right)^k.
$$
So for any $k\ge k_0$
$$
\sum^\infty_{s=k}\sum^\infty_{t=s}|a_t|=\sum^\infty_{s=k}\sum^\infty_{t=s}\tfrac{3}{4}\tfrac{1}{2^{s}}=3\tfrac{1}{2^{k}}<
\left(\tfrac{9}{10}\right)^k\left(1-\left(\tfrac{5}{8}\right)^k\right)^k\left(\tfrac{5}{8}\right)^k.
$$
  It is easy to see that $x_{n}=\left( -1\right) ^{n}$ is the bounded solution to \eqref{prob}.
\end{exm}
Using the same technique we get the following result.
\begin{thm}
Assume that
\begin{itemize}
\item[$(H_0)$] $\tau,\sigma\in\N\cup\{0\}$, $\tau>\sigma$,
\item[$(H_{fb})$] $f$ is a locally Lipschitz function and $f$ is a bounded function with the constant $P$;
\item[$(H_{sb})$] there exists decreasing sequence $(w_k)_{k\in\N}\subset(1,\infty)$ with $\sum_{k=1}^{\infty}(w_k-1)<\infty$ such that
$$
\sum\limits_{s=k}^{\infty }\left\vert\frac{1}{r_{s}}\right\vert\sum\limits_{t=s}^{\infty }\left\vert (a_{t}\right\vert P+\left\vert b_{t}\right\vert )\in O\left(\left(w_k-1\right)w_k^{-k}\right);
$$ 
\item[$(H^1_{q=1})$] $q_n>1,\ n\in\N$, $\lim\limits_{n\to\infty}q_n=1$.
\end{itemize}
Then, there exists a bounded full solution to \eqref{problem}.
\end{thm}

\begin{rem}
Analogous theorems we get if we change the assumption $(H_{q=1})$ or $(H^1_{q=1})$ to one of assumptions
\begin{itemize}
\item[$(H_{q=-1})$] $q_n\in(-1,0),\ n\in\N,\quad  \lim\limits_{n\to\infty}q_n=-1$, $\sup\limits_{n\in\N}q_n>0$.
\item[$(H^1_{q=-1})$] $q_n<-1,\ n\in\N,\quad  \lim\limits_{n\to\infty}q_n=-1.$
\end{itemize} 
\end{rem}

\section{The existence of $l^p$-solution}
In this section, we also assume $\tau\in \N\cup\{0\}$, $\sigma\in\Z$, $a, b, q:{\mathbb{N}}\rightarrow {\mathbb{R}}$, $r:{\mathbb{N}}\rightarrow {\mathbb{R}}\setminus \{0\}$ and $f:{\mathbb{R}}\rightarrow \mathbb{R}$. 
\begin{thm}\label{lp}
Assume that
\begin{itemize}
\item[$(H_p)$] $p\ge 1$;
\item[$(H_{fl})$] $f$ is a locally Lipschitz function; 
\item[$(H_{sp})$] $\sum\limits_{n=1}^{\infty }\left(\sum\limits_{s=n}^{\infty } \left\vert\frac{1}{r_{s}}\right\vert\sum\limits_{t=s}^{\infty }\left\vert a_{t}\right\vert\right)^p <+\infty,$\quad $\sum\limits_{n=1}^{\infty }\left(\sum\limits_{s=n}^{\infty } \left\vert\frac{1}{r_{s}}\right\vert\sum\limits_{t=s}^{\infty }\left\vert b_{t}\right\vert\right)^p <+\infty,$
\item[$(H_{qp})$] $\sup\limits_{n\in\N}|q_n|=q^\star\in(0,2^{1-p})$.
\end{itemize} 
Then, equation \eqref{problem} 
possesses a $l^p$-solution.
\end{thm}

\begin{proof}
From the continuity of $f$ on $[-1,1]$ we get the existence of $W>0$ such that
$$
|f(x)|\leq W, \ \textrm{for} \ x\in[-1, 1].
$$
The assumption $(H_{sp})$ implies there exists $n_0>\beta:=\max\{\tau,\sigma\}$ such that
\begin{equation}\label{szereg-p}
4^{p-1}\left[W^p\sum^{\infty}_{n=n_0}\left(\sum^{\infty}_{s=n}\left\vert\frac{1}{r_s}\right\vert\sum^{\infty}_{t=s}|a_t|\right)^p+\sum^{\infty}_{n=n_0}\left(\sum^{\infty}_{s=n}\left\vert\frac{1}{r_s}\right\vert\sum^{\infty}_{t=s}|b_t|\right)^p\right]<1-2^{p-1}q^\star.
\end{equation}
We consider Banach space $l^p$ and its subset
$$
A_{n_0}=\left\{x=(x_n)_{n\in\N}\in l^p: x_1=\ldots=x_{n_0+\beta-1}=0, \ \norm{x}_{l^p}\le 1\right\}.
$$ 
Observe that $A_{n_0}$ is a nonempty, bounded, convex and closed subset of $l^p$.
\\Define two mappings $T_1,T_2\colon l^p\rightarrow l^p$ as follows 
$$
(T_1x)_{n}=
\begin{cases}
0, &\text{for}\ 1\leq n<n_{0}+\beta
\\ 
-q_{n}x_{n-\tau}, &\text{for}\ n\geq n_{0}+\beta
\end{cases}
$$
$$
(T_2x)_{n}=
\begin{cases}
0, &\text{for}\ 1\leq n<n_{0}+\beta
\\ 
\sum\limits_{s=n}^{\infty }\frac{1}{r_{s}}\sum\limits_{t=s}^{\infty }\left( a_{t}f(x_{t-\sigma})+b_t \right), &\text{for}\ n\geq n_{0}+\beta.
\end{cases}
$$
Now, we prove that assumptions of Theorem \ref{Kras} - Krasnoselskii's fixed point. 
\\Firstly, we show that $T_1x+T_2y\in A_{n_0}$ for $x,y\in A_{n_0}$. Let $x,y\in A_{n_0}$. For $n<n_{0}+\beta$ $(T_1x+T_2y)_n=0$. Using twice the classical inequality
$$
(x+y)^p\le 2^{p-1}(x^p+y^p) \ \mbox{for} \ x,y\ge 0, \ p\ge1
$$
for $n\geq n_{0}+\beta$ we get
\begin{align*}
&\left\vert(T_1x+T_2y)_n\right\vert^p\leq 2^{p-1}\left[(q^\star)^p\left\vert x_{n-\tau}\right\vert^p +\left(\sum^{\infty}_{s=n}\left\vert\tfrac{1}{r_s}\right\vert\sum^\infty_{t=s}\left(|a_t|W+|b_t|\right)\right)^p\right], 
\\
&\le2^{p-1}\Bigg[q^\star\left\vert x_{n-\tau}\right\vert^p + 2^{p-1}\Bigg(W^p\left(\sum^{\infty}_{s=n}\left\vert\frac{1}{r_s}\right\vert\sum^{\infty}_{t=s}|a_t|\right)^p+\left(\sum^{\infty}_{s=n}\left\vert\frac{1}{r_s}\right\vert\sum^{\infty}_{t=s}|b_t|\right)^p\Bigg)\Bigg]
\\
&\le 2^{p-1}q^\star\left\vert x_{n-\tau}\right\vert^p + 4^{p-1}\Bigg[W^p\left(\sum^{\infty}_{s=n}\left\vert\frac{1}{r_s}\right\vert\sum^{\infty}_{t=s}|a_t|\right)^p+\left(\sum^{\infty}_{s=n}\left\vert\frac{1}{r_s}\right\vert\sum^{\infty}_{t=s}|b_t|\right)^p\Bigg].
\end{align*}
By \eqref{szereg-p} we obtain that
\begin{align*}
&\norm{T_1x+T_2y}^p_{l^p}\le 2^{p-1}q^\star\norm{x}^p_{l^p}+4^{p-1}\Bigg[W^p\sum^\infty_{n=n_0+\beta}\left(\sum^{\infty}_{s=n}\left\vert\frac{1}{r_s}\right\vert\sum^{\infty}_{t=s}|a_t|\right)^p
\\
&+\sum^\infty_{n=n_0+\beta}\left(\sum^{\infty}_{s=n}\left\vert\frac{1}{r_s}\right\vert\sum^{\infty}_{t=s}|b_t|\right)^p\Bigg]\le 1.
\end{align*}
It is easy to see that
$$
\norm{T_1x-T_1y}_{l^p}\leq q^\star \norm{x-y}_{l^p},\ \textrm{for} \ x,y\in A_{n_0},
$$
so that $T_1$ is a contraction. 
\\To prove the continuity of $T_2$, we note assumption $(H_{fl})$ implies that $f$ is a Lipschitz function on $[-1,1]$ which means that exists $L>0$ such that
$$
\abs{f(u)-f(v)}\leq L\abs{u-v},\ \textrm{for} \ u,v\in [-1,1].
$$
Hence for $x, y\in A_{n_0}$ and $n\ge n_0+\beta$
\begin{align*}
&\abs{(T_2x-T_2y)_n}^p\leq \left( \sum\limits_{s=n}^{\infty} \left\vert\frac{1}{r_{s}}\right\vert\sum\limits_{t=s}^{\infty }|a_{t}|L\abs{x_{t-\sigma}-y_{t-\sigma}} \right)^p 
\\
&\le L^p\left(\sum^\infty_{s=n}\left\vert\frac{1}{r_{s}}\right\vert\sum\limits_{t=s}^{\infty }|a_{t}|\right)^p\norm{x-y}^p_{l_p}.
\end{align*}
Combine above with $(H_{sp})$ we get that
$$
\norm{T_2x-T_2y}_{l_p}\leq L\left[\sum\limits_{n=0}^{\infty}\left(\sum\limits_{s=n}^{\infty} \left\vert\frac{1}{r_{s}}\right\vert\sum\limits_{t=s}^{\infty }\left\vert a_{t}\right\vert \right)^p\, \right]^{1/p}\norm{x-y}_{l_p}, \ \textrm{for} \ x,y\in A_{n_0}.
$$
Actually, we prove that $T_2$ is a Lipschitz operator.
\\In an analogous way we get for $x\in A_{n_0}$ and $n\ge n_0+\beta$ 
$$
\abs{(T_2x)_n}^p\le 2^{p-1}\left[W^p\left(\sum^\infty_{s=n}\left\vert\tfrac{1}{r_s}\right\vert\sum^\infty_{t=s}|a_t|\right)^p+\left(\sum^\infty_{s=n}\left\vert\tfrac{1}{r_s}\right\vert\sum^\infty_{t=s}|b_t|\right)^p\right]
$$
and hence by $(H_{sp})$
\begin{align*}
&\lim_{l\to\infty}\sup_{x\in A_{n_0}}\sum^\infty_{n=l}|(T_2x)_n|^p \le  
\\
&\lim_{l\to\infty}\sum^\infty_{n=l}2^{p-1}\left[W^p\left(\sum^\infty_{s=n}\left\vert\tfrac{1}{r_s}\right\vert\sum^\infty_{t=s}|a_t|\right)^p+\left(\sum^\infty_{s=n}\left\vert\tfrac{1}{r_s}\right\vert\sum^\infty_{t=s}|b_t|\right)^p\right]=0,
\end{align*}
which means that $T_2(A_{n_0})$ is relatively compact subset of $l^p$.
\\From the Theorem \ref{Kras} we get that there exists $x=(x_n)_{n\in\N_1}$ the fixed point of $T_1+T_2$ on $A_{n_0}$. Applying operator $\Delta$ to both sides of the above equation and multiplying by $r_n$ and applying operator $\Delta$ second time for $n\ge n_0+\beta$ we get $x=(x_n)_{n\in\N_{n_0+\beta}}$ is the $l^p$-solution to \eqref{problem}.
\end{proof}

\begin{rem}
It is worth mention that for $p=1$ assumption $(H_{sb})$ implies assumption $(H_{sp})$.
\end{rem}

Now, we present an example of equation  for which our method can be applied.

\begin{exm}
Let us consider the following problem 
\begin{equation}\label{prob-1}  
\Delta\left( \left( -1\right) ^{n}\Delta \left( x_{n}+q_nx_{n-3}\right)\right)=2^{-n}f(x_{n-\sigma})+\tfrac{1}{n(n+1)(n+2)(n+3)},\ n\ge 3
\end{equation}
with $\tau=3$, $\sigma=1$, $r_{n}=(-1)^{n}$, $a_{n}=2^{-n}$, $b_n=\tfrac{1}{n(n+1)(n+2)(n+3)}$, $n\ge 1$ and any $(q_n)$,  $\sup|q_{n}|<1$, $f\in C^1(\R)$. Note
\begin{align*}
&\sum^{\infty}_{n=1}\sum^{\infty}_{s=n}\left|\tfrac{1}{r_s}\right|\sum^{\infty}_{t=s}|a_t|=4,
\\
&\sum^{\infty}_{n=1}\sum^{\infty}_{s=n}\left|\tfrac{1}{r_s}\right|\sum^{\infty}_{t=s}|b_t|=\sum^{\infty}_{n=1}\sum^{\infty}_{s=n}\sum^{\infty}_{t=s}\tfrac{1}{t(t+1)(t+2)(t+3)}
\\
&=\sum^{\infty}_{n=1}\sum^{\infty}_{s=n}\tfrac{1}{4s(s+1)(s+2)}=\sum^{\infty}_{n=1}\tfrac{1}{12n(n+1)}<\infty
\end{align*}
which means that assumptions of the Theorem \ref{zb-1} are fulfilling with $p=1$. Hence \eqref{prob-1} has a $l^1$-solution. It is obvious that this $l^1$-solution is a $l^p$-solution for any $p>1$.
\end{exm}

\begin{cor}\label{lp-drugie}
Assume that
\begin{itemize}
\item[$(H_p)$] $p\ge 1$;
\item[$(H_{fl})$] $f$ is a locally Lipschitz function; 
\item[$(H'_{sp})$] $\sum\limits_{n=1}^{\infty }\left(\sum\limits_{s=n}^{\infty } \left\vert\frac{1}{r_{s}}\right\vert\sum\limits_{t=\sigma}^{s-1}\left\vert a_{t}\right\vert\right)^p <+\infty,$\quad $\sum\limits_{n=1}^{\infty }\left(\sum\limits_{s=n}^{\infty } \left\vert\frac{1}{r_{s}}\right\vert\sum\limits_{t=\sigma}^{s-1}\left\vert b_{t}\right\vert\right)^p <+\infty,$
\item[$(H_{qp})$] $\sup\limits_{n\in\N}|q_n|=q^\star\in(0,2^{1-p})$.
\end{itemize} 
Then, equation \eqref{problem} possesses a bounded solution.
\end{cor}

\begin{cor}
If in Theorem \ref{lp} or Corollary \ref{lp-drugie} we additionally assume
\begin{itemize}
\item[$(H'_{0})$] $\tau,\sigma\in\N\cup\{0\}$, $\tau>\sigma$ and $q_n\neq0$, \ $n\in\N$.
\end{itemize} 
Then, equation \eqref{problem} possesses full $l^p$-solution.
\end{cor}


\begin{thebibliography}{99}

\bibitem{Agarwal} R.P. Agarwal, "Difference Equations and Inequalities, second edition," Marcel Dekker, New York, 2000.

\bibitem{Agarwal2003} R.P. Agarwal, S.R. Grace and D. O'Regan, {\it Nonoscillatory solutions for discrete equations,} Comput. Math. Appl., {\bf 45} (2003), 1297--1302.

\bibitem{costara} C. Costara and D. Popa, Exercises in Functional Analysis, Kluwer Academic Press, 2003.

\bibitem{Cheng1999} S.S. Cheng, {\it Existence of nonoscillatory solutions of a second-order linear neutral difference equation,} Appl. Math. Lett. {\bf 12} (1999),71--78.

\bibitem{Cheng} S.S. Cheng and W. T. Patula, {\it An existence theorem for a nonlinear difference equation,} Nonlinear Anal. {\bf 20} (1993),1297--1302.

\bibitem{Dosly} O. Do\u{s}l\'{y}, J. Graef, J. Jaro\u{s}, {\it Forced oscillation of second order linear and half--linear difference equations,} Proc. Amer. Math. Soc. {\bf 131} (2002), 2859--2867.

\bibitem{Elaydi} S.N. Elaydi, "An Introduction to Difference Equations, Third edition. Undergraduate Texts in Mathematics", Springer, New York, 2005. 

\bibitem{EP2008} K. Ey, C. P\"otzsche, {\it Asymptotic behavior of recursions via fixed point theory}, J. Math. Anal. Appl. {\bf337} (2008), 1125--1141.

\bibitem{Galewski} M. Galewski, R. Jankowski, M. Nockowska-Rosiak and E. Schmeidel, {\it On the existence of bounded solutions for nonlinear second order neutral difference equations,} Electron. J. Qual. Theory Differ. Equ., {\bf 72} (2014), 1--12.

\bibitem{Gou} Z. Gou and  M. Liu, {\it Existence of non-oscillatory solutions for a higher-order nonlinear neutral difference equation,} Electron. J. Diff. Equ., {\bf 146} (2010), 1--7.

\bibitem{JS} R. Jankowski and E. Schmeidel, {\it Almost oscillation criteria for second order neutral difference equation with quasidifferences,} Int. J. Difference Equ. {\bf 9} (2014), 77-86.

\bibitem{Jinfa} C. Jinfa, {\it Existence of a nonoscillatory solution of a second-order linear neutral difference equation,} Appl. Math. Lett. {\bf 20} (2007) 892--899.

\bibitem{Lalli1994} B.S. Lalli and S.R. Grace, {\it Oscillation theorems for second order neutral difference equations,} Appl. Math. Comput., {\bf 62} (1994), 47--60.

\bibitem{Lalli1992} B.S. Lalli and B.G. Zhang, {\it  On existence of positive solutions and bounded oscillations for neutral difference equations,} J. Math. Anal. Appl., {\bf 166} (1992), 272--287.

\bibitem{Liu2009} Z. Liu, Y. Xu and S. M. Kang, {\it Global solvability for second order nonlinear neutral delay difference equation,} Comput. Math. Appl., {\bf 57} (2009), 587--595.

\bibitem{Liu2011} Z. Liu, L. Zhao, S. M. Kang and J. S. Ume, {\it  Existence of uncountably many bounded positive solutions for second order nonlinear neutral delay difference equations,} Comput. Math. Appl., {\bf 61} (2011), 2535--2545.

\bibitem{Lou} J.W. Luo and D.D. Bainov, {\it Oscillatory and asymptotic behavior of second--order neutral difference equations with maxima,} J. Comput. Appl. Math., {\bf 131} (2001), 333--341.

\bibitem{Meng} Q. Meng and J. Yan, {\it Bounded oscillation for second-order nonlinear neutral difference equations in critical and noncritical states,} J. Comput. Appl., {\bf 57} (2009), 587--595.

\bibitem{Migda2005} J. Migda and M. Migda, {\it Asymptotic properties of solutions of second order neutral difference equations,} Nonlinear Anal. {\bf 63} (2005), 789--799.

\bibitem{Migda2014} J. Migda, {\it Approximate solutions of difference equations,} Electron. J. Qual. Theory Differ. Equ., {\bf 13} (2014), 1--26.

\bibitem{N2016} M. Nockowska-Rosiak, {\it Existence of uncountably many bounded positive solutions to higher order nonlinear neutral delay difference equations}, Adv. Difference Equ (2016), 2016:198, doi 10.1186/s13662-016-0923-2.

\bibitem{PS2001} E. Petropoulou, P. Siafarikas, {\it Bounded solutions and asymptotic stability of nonlinear difference equations in the complex plane. II}, Comput. Math. Appl. {\bf 42} (3–5) (2001) 427–452.


\bibitem{PS2004} E. Petropoulou, P. Siafarikas, {\it A functional-analytic method for the study of difference equations}, Adv. Difference Equ. {\bf 3} (2004) 237-–248.

\bibitem{PS2006} E. Petropoulou, P. Siafarikas, {\it Existence of complex $l_2$ solutions of linear delay systems of difference equations}, J. Differ. Equations Appl., {\bf 11}(1) (2006), 49--62.


\bibitem{P2009} C. P\"otzsche, {\it A functional-analytical approach to the asymptotic of recursions}, Proc. Amer. Math. Soc. {\bf 137}(10) (2009), 3297--3307.

\bibitem{Rudin} W. Rudin, "Functional Analysis. Second Edition." McGraw-Hill, Inc., Singapore, 1991.

\bibitem{S1} S.H. Saker, {\it New oscillation criteria for second--order nonlinear neutral delay difference equations,}  Appl. Math. Comput., {\bf 142} (2003), 99--111.

\bibitem{S2} S.H. Saker, {\it Oscillation theorems of nonlinear difference equations of second order,} Georgian Math. J., {\bf 10} (2003), 343--352.

\bibitem{S3} S.H. Saker, {\it Oscillation of second--order perturbed nonlinear difference equations,} Appl. Math. Comput., {\bf 144} (2003), 305--324.

\bibitem{Schmeidel2013} E. Schmeidel, {\it An application of measures of noncompactness in investigation of boundedness of solutions of second order neutral difference equations,} Adv. Difference Equ.  {\bf 91} (2013), 9 pp.

\bibitem{Schmeidel2012} E. Schmeidel and Z. Zb\c{a}szyniak, {\it An application of Darbo's fixed point theorem in the investigation of periodicity of solutions of difference equations,} Comput. Math. Appl. {\bf 64} (2012), 2185--2191.

\bibitem{TKP1} E. Thandapani, N. Kavitha and S. Pinelas, {\it Oscillation criteria for second--order nonlinear neutral difference equations of mixed type,} Adv. Difference Equ. 2012, DOI: 10.1186/1687--1847--2012--4, 10 pp. 

\bibitem{Zeidler} E. Zeidler, "Nonlinear Functional Analysis and its Applications I. Fixed-Point Theorems", Springer-Verlag, New York, 1986. 

\end{thebibliography}
\end{document}